\documentclass[11pt]{article}

\usepackage{enumerate}
\usepackage{amsmath}
\usepackage{amssymb,latexsym}
\usepackage{amsthm}
\usepackage{color}
\usepackage{graphicx}
\usepackage{amscd}
\usepackage{hyperref}
\usepackage[normalem]{ulem}

\title{Partially scattered linearized polynomials and rank metric codes}
\author{Giovanni Longobardi - Corrado Zanella}
\date{}

\newcommand{\aaa}{\mathbf a}

\newcommand{\cC}{{\mathcal C}}

\newcommand{\cN}{{\mathcal N}}

\newcommand{\cF}{{\mathcal F}}

\newcommand{\cD}{{\mathcal D}}

\newcommand{\cS}{{\mathcal S}}

\newcommand{\F}{{\mathbb F}}
\newcommand{\V}{{\mathbb V}}
\newcommand{\Fq}{{\F_q}}

\newcommand{\Fqn}{\F_{q^n}}
\newcommand{\Fqt}{{\F_{q^t}}}

\newcommand{\la}{\langle}
\newcommand{\ra}{\rangle}
\renewcommand{\mod}{\hbox{{\rm mod}\,}}

\newcommand{\Rt}{\operatorname{R-}q^t\operatorname-}
\newcommand{\Lt}{\operatorname{L-}q^t\operatorname-}

\newtheorem{theorem}{Theorem}[section]

\newtheorem{proposition}[theorem]{Proposition}

\DeclareMathOperator{\tr}{Tr}

\DeclareMathOperator{\PG}{{PG}}

\DeclareMathOperator{\GL}{{GL}}

\DeclareMathOperator{\Gal}{Gal}
\DeclareMathOperator{\rk}{rk}
\DeclareMathOperator{\N}{N}
\DeclareMathOperator{\p}{p}
\DeclareMathOperator{\End}{End}
\DeclareMathOperator{\im}{im}

\theoremstyle{definition}
\newtheorem{definition}[theorem]{Definition}

\newtheorem{remark}[theorem]{Remark}

\begin{document}

\relax

\maketitle

\begin{abstract}
A linearized polynomial $f(x)\in\Fqn[x]$ is called scattered if  for any $y,z\in\Fqn$, the condition
$zf(y)-yf(z)=0$ implies that $y$ and $z$ are $\Fq$-linearly dependent.
In this paper two generalizations of the notion of a scattered linearized polynomial
are provided and investigated.
Let $t$ be a nontrivial positive divisor of $n$.
By weakening the property defining a scattered linearized polynomial, 
$\Lt$partially scattered and $\Rt$partially scattered linearized polynomials are introduced
in such a way that the scattered linearized polynomials are precisely those which are both
$\Lt$ and $\Rt$partially scattered.
Also, connections between partially scattered polynomials, linear sets and rank metric codes are exhibited.
\end{abstract}



\section{Introduction} \label{s:uno}

Let $U$ be an $\Fq$-linear subspace of an $r$-dimensional vector space $V$ over $\Fqn$.
The following set of points in the projective space $\PG(V,\Fqn)\cong\PG(r-1,q^n)$:
\[L_U=\{\la v\ra_{\Fqn}\colon v\in U\setminus\{0\}\},\] consisting of all points
represented by nonzero vectors in $U$, is called an 
\emph{$\Fq$-linear set} of \emph{rank} $m=\dim_\Fq U$.
The rank of a linear set is not always uniquely defined.
For example, any $\Fq$-linear set of rank greater than $(r-1)n$ in $\PG(r-1,q^n)$ trivially
coincides with $\PG(r-1,q^n)$.
The \emph{weight} of a point $P=\la v\ra_{\Fqn}$ of $\PG(r-1,q^n)$ with respect to $L_U$
is $w(P)=\dim_{\Fq}(\la v\ra_{\Fqn}\cap U)$.

A  \emph{canonical $\Fq$-subgeometry} in $\PG(m-1,q^n)$ is
an $\Fq$-linear set of rank $m$, say ${\Sigma}$, 
such that $\la {\Sigma}\ra=\PG(m-1,q^n)$
(\footnote{In this paper, $\la A\ra_F$ denotes the $F$-linear span of the set of vectors 
$A$, $F$ a field,
whereas $\la B\ra$ (without subscript) denotes the projective span of the set of points $B$
in a projective space.}).
Lunardon and Polverino in \cite{LuPo2004} (see also \cite{LuPoPo2002}) showed that every linear
set is a projection of a subgeometry.
More precisely, let $\Sigma$, $\Gamma$ and $\Lambda$ be 
a canonical $\Fq$-subgeometry,
an $(m-r-1)$-subspace and an $(r-1)$-subspace
of $\Sigma^*=\PG(m-1,q^n)$, respectively, such that $\Sigma\cap\Gamma=\emptyset=\Lambda\cap\Gamma$.
Let $\p_{\Gamma,\Lambda}$ be the
\textit{projection from $\Gamma$ into $\Lambda$} defined by
$\p_{\Gamma,\Lambda}(P)=\langle \Gamma,P\rangle \cap \Lambda$ 
for each point $P \in \Sigma^*\setminus\Gamma$.
We call $\Gamma$ and $\Lambda$ the \emph{vertex} (or \emph{center}) and the \emph{axis} of the projection, respectively.
Then $L=\p_{\Gamma,\Lambda}(\Sigma)$ is an $\F_q$-linear set of $\Lambda$ of 
rank $m$ and $\langle L \rangle=\Lambda$.
Conversely, any $\Fq$-linear set spanning $\Lambda$ can be obtained in this way. This includes the canonical subgeometries as projections from the empty vertex.

Let  $V$ be an $r$-dimensional vector space over $\Fqn$.
The projective space
$\PG (V,\Fq)\cong\PG(rn-1,q)$ is called the projective space obtained from
$\PG (V,\Fqn)\cong\PG(r-1,q^n)$ by \emph{field reduction}.
The point set of $\PG(V,\F_{q^n})$
induces a partition $\cS$ of $\PG (V,\Fq)$ into $(n-1)$-subspaces; such $\cS$ is called
the \emph{normal $(n-1)$-spread of $\PG (V,\Fq)$ induced by  $\PG (V,\Fqn)$} 
by field reduction.

Any $m$-dimensional $\Fq$-subspace $U$ of $V$ determines an
$(m-1)$-dimensional subspace, say $\tilde U$, of $\PG (V,\Fq)$.
Such $\tilde U$ is called \emph{scattered with respect to $\cS$} if any element of 
$\cS$
meets $\tilde U$ in at most one point. 
In this case, the related linear set $L_U$ is called a \emph{scattered linear set}.
Equivalently, $L_U$ is scattered when all its points have weight one.
The subspace $\tilde U$ is called \emph{maximum scattered}, and $L_U$ is called 
\emph{maximum scattered linear set} if any $m$-dimensional subspace of the projective
space $\PG(V,\Fq)$ is not scattered.
If $rn$ is even, the dimension $m-1$ of a maximum scattered subspace with respect to $\cS$
satisfies $m=rn/2$ \cite{BL2000}.
The known results concerning the non-even values of $rn$ are surveyed in \cite{OlgaFer}.

An $\Fq$-\emph{linearized polynomial}, or \emph{$q$-polynomial}, over $\Fqn$ is a polynomial of the form
$f(x) =\sum_{i=0}^{k}a_i x^{q^i}\in\Fqn[x]$, $k \in \mathbb{N}$. 
If $a_k \neq 0$, the integer $k$ is called the \emph{$q$-degree} of $f(x)$.
It is well known that any linearized polynomial defines an endomorphism of $\F_{q^n}$, 
when $\F_{q^n}$ is 
regarded as an $\F_q$-vector space and, vice versa, each element of $\End_{\F_q}(\F_{q^n})$	can be represented 
as a unique linearized polynomial over $\F_{q^n}$ of $q$-degree less than $n$; see \cite{lidl}.
For a $q$-polynomial $f(x) = \sum_{i=0}^{n-1}a_i x^{q^i}$ 
over $\Fqn$, let $D_f$ denote the associated \textit{Dickson matrix} (or $q$-\textit{circulant matrix} )
\begin{equation}\label{dickson}
D_f=
\begin{pmatrix}

a_0 & a_1 &\ldots & a_{n-1}\\
a^q_{n-1} & a^q_0 & \ldots & a^q_{n-2}\\
\vdots & \vdots & \vdots & \vdots \\
a_1^{q^{n-1}} & a_{2}^{q^{n-1}} & \ldots & a_0^{q^{n-1}}
\end{pmatrix}.
\end{equation}

The rank of the matrix $D_f$ is the rank of the  $\Fq$-linear map $f(x)$, see \cite{WuLiu}.\\
Any $\Fq$-linear set of rank $n$ in the projective  line 
$\PG(1,q^n)$ is projectively equivalent
to $L_f=L_{U_f}$ where $U_f=\{(x,f(x))\colon x\in\Fqn\}$ with $f(x)\in\Fqn[x]$ a $q$-polynomial.
The polynomial $f(x)$ is called a \emph{scattered $q$-polynomial}
if $L_f$ is scattered.
That is, $f(x)$ is scattered if and only if for any $y,z\in\Fqn$, the condition
$zf(y)-yf(z)=0$ implies that $y$ and $z$ are $\Fq$-linearly dependent.

An $\Fq$-linear \emph{rank metric code} (or \emph{RM-code}) is an $\Fq$-subspace $\cC$ of 
$\Fq^{\ell\times n}$, endowed with the metric $d(A,B)=\rk(A-B)$ for any $A,B\in\cC$.
If $|\cC| \geq 2$, the \emph{minimum distance} of $\cC$ is
$$d=\min\{d(A,B)\colon A,B\in\cC,\,A\neq B\}.$$
The \emph{Singleton-like bound} states that if $k=\dim_\Fq\cC$, then
\[
k\le\max\{\ell,n\}(\min\{\ell,n\}-d+1).
\]
If the equality in the Singleton-like bound holds, then $\cC$ is a \emph{maximum rank distance code},
or \emph{MRD-code} with parameters $(\ell,n,q;d)$.\\
The \emph{adjoint code} of an RM-code $\cC \subset \F_q^{\ell \times n}$ is the set 
$\cC^\top=\{C^t\colon C\in\cC\}\subset \F_q^{n\times \ell}$,
where $C^t$ denotes the transpose of the matrix $C$.
Take into account the symmetric bilinear form on $\F_q^{\ell \times n}$ 
$$ \langle M, N \rangle  = \tr(MN^t).$$
The \textit{Delsarte dual code} of $\cC$ is
\begin{equation}\label{delsarte-dual}
\cC^\perp= \{N \in \F_q^{\ell \times n} \,:\,
\langle M, N \rangle =0,  \,\, \forall M \in  \cC\}.
\end{equation}
In \cite{delsarte}, Delsarte proved that if $\cC \subset  \F_q^{\ell \times n}$ is an $\Fq$-linear MRD-code with dimension $k$ and  $d > 1$, then $\cC^\perp$ 
is an MRD-code of dimension $\ell n - k$.\\
Two $\Fq$-linear codes $\cC,\cC'\subset \F_q^{\ell \times n}$ are called \textit{equivalent} if
there exist $P\in\GL(\ell,q)$, $Q \in \GL(n, q)$ and a field
automorphism $\sigma$ of $\Fq$ such that 
$$\cC' = \{PC^\sigma Q \colon C \in \cC\}.$$
Furthermore, $\cC$ and $\cC'$ are \emph{weakly equivalent} if $\cC$ is equivalent to
$\cC'$ or to $(\cC')^\top$.
Finally, the
\textit{left} and \textit{right idealisers} $L(\cC)$ and $R(\cC)$ of an RM-code 
$\cC\subset \F_q^{\ell \times n}$
 are defined as the sets
\[L(\cC) = \{X \in \F_q^{\ell\times\ell} \colon X C  \in \cC \,\,\textnormal{for all} \,\, C \in \cC\},\]
\[R(\cC) = \{Y \in \F_q^{n\times n} \colon  C Y  \in \cC \,\,\textnormal{for all}\,\, C \in \cC\},\]
respectively. These two concepts were introduced in \cite{liebhold_automorphism_2016}, and in \cite{LTZ2} with different names. 
Since the left (resp.\ right) idealisers of two equivalent RM-codes are equivalent,
they turn out to be an useful tool in investigating the equivalence issue among RM-codes. 

\begin{proposition}[\cite{LTZ2}, Propositions 4.1 and 4.2, Theorem 5.4 and Corollary 5.6]\label{MRD-idealisers}  Let $\cC$ be an $\Fq$-linear RM-code of $\F^{\ell \times n}_q$. 
	The following statements hold:
	\begin{enumerate}[(a)]
		\item  $L(\cC^\top)=R(\cC)^\top$ and $R(\cC^\top)=L(\cC)^\top$
		\item $L(\cC^\perp)=L(\cC)^\top$  and $R(\cC^\perp) =R(\cC)^\top$
	\end{enumerate}
	Next assume that $\cC$ is an $\Fq$-linear MRD-code of $\F^{\ell \times n}_q$ with minimum distance $d > 1$.
	If $\ell \leq n$, then $L(\cC)$ is a finite field with $|L(\cC)| \leq  q^\ell$. If $\ell \geq n$, then $R(\cC)$ is
	a finite field with $|R(\cC)| \leq q^n$. In particular, when $\ell = n$ then $L(\cC)$ and $R(\cC)$
	are both finite fields.
\end{proposition}

The reader is referred to \cite{OlgaFer} for  more generalities on the rank metric codes,
as well as for other connections between scattered linear sets and MRD-codes.
See also \cite{Lunsur,Sh}.

In this paper two classes of $q$-polynomials in $\Fqn[x]$, $n$ a non-prime integer, are dealt
with, both containing all scattered $q$-polynomials of $\Fqn[x]$.
Such polynomials are defined in Sections \ref{elle} and \ref{erre}, and called 
$\Lt$partially scattered and $\Rt$partially scattered
$q$-polynomials, with $t\neq1$ a proper divisor of $n$.
The polynomials belonging to the intersection of both classes are precisely the scattered
$q$-polynomials.
In Theorem \ref{rflf}, the relations between the linear sets $L_f$, $\tilde U_f$ and
$R_f=\{\la v\ra_{\Fqt}\colon v\in U,\,v\neq0\}\subseteq\PG(2(n/t)-1,q^t)$ are investigated,
where $f(x)$ is an $\Lt$ or $\Rt$partially scattered linearized polynomial.
Some algebraic conditions for partially scattered polynomials are described, as well as
general techniques of construction and examples.
In Section \ref{s:tre} rank metric codes are defined related to
$\Rt$partially scattered $q$-polynomials.
In Theorem \ref{R-ps_MRDsquare} such codes are proved to be MRD.
Furthermore, their left and right idealisers are investigated.

\section{\texorpdfstring{Partially scattered $q$-polynomials}{Partially scattered q-polynomials}}%
\label{s:due}

Throughout this paper $f(x)$ denotes a $q$-polynomial of the form $\sum_{i=0}^{n-1}a_ix^{q^i}\in\F_{q^n}[x]$,
where $n=tt'$ and $t,t'\in\mathbb N\setminus\{0,1\}$.

\subsection{Definition and geometric meaning}

\begin{definition}\label{elle}
The $q$-polynomial $f(x)$ is \emph{$\Lt$partially scattered} if for any $y,z\in\F_{q^n}^*$,
\begin{equation}\label{e:elle}
\frac{f(y)}y=\frac{f(z)}z\quad\Longrightarrow\quad\frac yz\in\F_{q^t}.
\end{equation}
\end{definition}
\begin{definition}\label{erre}
The $q$-polynomial $f(x)$ is \emph{$\Rt$partially scattered} if for any $y,z\in\F_{q^n}^*$,
\begin{equation}\label{e:erre}
\frac{f(y)}y=\frac{f(z)}z\ \land\ \frac yz\in\F_{q^t}\quad\Longrightarrow\quad\frac yz\in\F_{q}.
\end{equation}
\end{definition}
Clearly, if $f(x)$ is both L- and $\Rt$partially scattered, then $f(x)$ is scattered.

\begin{theorem}\label{rflf}
Let $U_f=\{(y,f(y))\colon y\in\Fqn\}$, $U_f^*=U_f\setminus\{(0,0)\}$, and consider the following sets:
\begin{enumerate}[A.]
\item$\tilde U_f=\{\la v\ra_{\Fq}\colon v\in U^*_f\}$; this is an 
$(n-1)$-dimensional subspace of the projective space $\PG(\Fqn^2,\Fq)\cong\PG(2n-1,q)$ 
associated with $V(\Fqn^2,\Fq)$;
\item $R_f=\{\la v\ra_{\Fqt}\colon v\in U^*_f\}$; this is an $\Fq$-linear set of rank $n$ in
the projective space $\PG(\Fqn^2,\Fqt)\cong\PG(2t'-1,q^t)$;
\item $L_f=\{\la v\ra_{\Fqn}\colon v\in U_f^*\}$; this is an $\Fq$-linear set of rank $n$ in
the projective line $\PG(\Fqn^2,\Fqn)\cong\PG(1,q^n)$.
\end{enumerate}
Then
\begin{enumerate}[(i)]
\item The size $\#R_f$ of $R_f$ equals $\#L_f$ if and only if $f(x)$ is
$\Lt$partially scattered.
\item $R_f$ is a scattered $\Fq$-linear set if and only if $f(x)$ is $\Rt$partially scattered.
In this case, the map $\mu:\la v\ra_{\Fq}\in\tilde U_f\mapsto\la v\ra_{\Fqt}\in R_f$ is bijective.
\item If $t=2$ and $f(x)$ is R-$q^2$-partially scattered, then $R_f$ is a canonical
$\Fq$-subgeometry of $\PG(2t'-1,q^2)$.
\end{enumerate}
\end{theorem}
\begin{proof}
The assertions in A., B., and C. directly follow from the definition of $\Fq$-linear set
and from the fact that $U_f$ is an $n$-dimensional subspace of $V(\Fqn^2,\Fq)$.

The map $\alpha:\la v\ra_{\Fqt}\mapsto\la v\ra_{\Fqn}$ is well-defined, hence
$\#L_f\le\#R_f$.
The equality holds if and only if $\alpha$ is one-to-one, that is, for any $y,z\in\Fqn^*$
the relation $\la(y,f(y))\ra_{\Fqn}=\la (z,f(z))\ra_{\Fqn}$ implies
$\la(y,f(y))\ra_{\Fqt}=\la (z,f(z))\ra_{\Fqt}$; this is equivalent to \eqref{e:elle}.
Hence (i) holds.

As regards (ii), $R_f$ is scattered if and only if every 
$\la (y,f(y))\ra_{\Fqt}\in R_f$ has weight one.
Assume $(z,f(z))\in U_f^*\cap\la(y,f(y))\ra_{\Fqt}$.
This implies $y/z\in\Fqt$ and $f(y)/y=f(z)/z$.
So, assuming \eqref{e:erre} yields
\begin{equation}\label{e:r-w}
  U_f\cap\la(y,f(y))\ra_{\Fqt}=\la(y,f(y))\ra_{\Fq}.
\end{equation}
Similarly if \eqref{e:r-w} holds for any $y\in\Fqn^*$, then $f(x)$ is $\Rt$partially scattered.
The map $\mu$ is bijective by definition of scattered linear set.

If $t=2$ and $f(x)$ is R-$q^2$-partially scattered, then $R_f$ is a one-to-one projection of an 
$\Fq$-canonical subgeometry $\Sigma$ of $\PG(2t'-1,q^2)$ \cite{LuPo2004}.
Since for any point $P$ in $\PG(2t'-1,q^2)$ there is at least one $(q+1)$-secant line to $\Sigma$,
the vertex of the projection is empty.
\end{proof}

By Theorem \ref{rflf}~(ii), if $f(x)$ is an $\Rt$partially scattered $q$-polynomial,
then $R_f$ is a maximum scattered $\Fq$-linear set of $\PG(2t'-1,q^t)$,
satisfying the assumptions of the following: 

\begin{proposition}
If $R$ is a maximum scattered $\Fq$-linear set in $\Omega=\PG(r-1,q^t)$, then $\la R\ra=\Omega$.
\end{proposition}
\begin{proof}
Let $R=L_U$ be associated to the $\Fq$-subspace $U$ of $V(\F_{q^t}^r,\Fq)$.
Assume $\dim_\Fq U=d$.
By definition, any $\Fq$-subspace of dimension greater than $d$ is not scattered with
respect to the $t$-spread $\cF=\{\la v\ra_{\Fqt}\colon v\in\F_{q^t}^r,\,v\neq0\}$  of $V(\F_{q^t}^r,\Fq)$.

Assume $\la R\ra\neq\Omega$; that is, there is a hyperplane $H$ of $V(\F_{q^t}^r,\F_{q^t})$
such that $\la v\ra_{\Fqt}\in R$ implies $v\in H$. In particular, $U\subseteq H$.

Take $z\in\F_{q^t}^r\setminus H$ and define $T=\la U,z\ra_\Fq$ that is a $(d+1)$-dimensional
subspace of  $V(\F^r_{q^t},\F_{q})$.

Let $F\in\cF$; then either $F\subseteq H$, or $F\cap H=\{0\}$.
If $F\subseteq H$, from $T\cap H=U$ one obtains $\dim_{\Fq}F\cap T\le 1$.
If $F\cap H=\{0\}$, then $F\cap U=\{0\}$,  and hence
\[
\dim_{\Fq}(F\cap T)=t+(d+1)-\dim_{\Fq}(F+T)\le t+(d+1)-\dim_{\Fq}(F+U)=1.
\]
Therefore, $T$ is scattered with respect to $\cF$, a contradiction.
\end{proof}

\begin{remark}
The $q$-polynomial $f(x)$ is $\Lt$partially scattered if and only if each element of 
the normal $(t'-1)$-spread $\cS'$ of $\PG(2t'-1,q^t)$ induced by
the point set of $\PG(\F_{q^n}^2,\F_{q^n})$ intersects $R_f$ in at most one point.
Defining a scattered point set with respect to a spread in the obvious way,
$R_f$ turns out to be scattered with respect to $\cS'$.
In particular, Theorem~\ref{rflf}~(iii) 
implies the existence of a scattered canonical subgeometry $R_f$ with respect to 
a normal spread.
\end{remark}

\subsection{Algebraic properties of partially scattered polynomials}

Some properties of partially scattered polynomials can be stated
in terms of the $q$-polynomials
$f_\rho(x)=f(\rho x)-\rho f(x)$, $\rho\in\Fqn^*$.
\begin{proposition}\label{p:frho}
\begin{enumerate}[(a)]
\item A $q$-polynomial $f(x)$ is $\Rt$partially scattered if and only if for any 
$\rho\in\Fqt\setminus\Fq$ the map $f_\rho(x)$ is bijective.
\item	A $q$-polynomial $f(x)$ is $L$-$q^t$-partially scattered if and only if for any $\rho \in 
\Fqn \setminus \F_{q^t}$ the map $f_{\rho}(x)$ is bijective.
\item	A $q$-polynomial $f(x)$ is scattered if and only if for any $\rho \in \Fqn \setminus \Fq$
the map $f_\rho(x)$ is bijective.
\end{enumerate}
\end{proposition}

\begin{proposition}\label{p:facile}
If $t$ and $t'$ are relatively prime, then any $\Lt$partially scattered $q$-polynomial
is also R-$q^{t'}$-partially scattered. 
\end{proposition}
\begin{proof}
By Proposition \ref{p:frho} and the hypothesis, $f_\rho$ is bijective for any $\rho \in 
\Fqn \setminus \F_{q^t}$, and in particular for $\rho\in\F_{q^{t'}}\setminus\Fq$.
\end{proof}
If the assumptions of Proposition \ref{p:facile} hold,
by Theorem \ref{rflf}~(ii)  the set
		$$M_f=\{\la (x,f(x))\ra_{\F_{q^{t'}}}\,:\,x \in \Fqn^*\}$$
		is a scattered $\F_q$-linear set  of $\PG(2t-1,q^{t'})$ of rank $n$.
Note that if $t'=2$,  as in Theorem \ref{rflf}~(iii), 
the $\F_q$-linear set $M_f$ is an $\F_q$-subgeometry of $\PG(2t-1,q^{2})$.

For any $\rho\in\Fqt$, the map $\Phi_\rho:\,f(x)\mapsto f_\rho(x)$ is an $\Fq$-endomorphism
of $\End_{\Fq}(\Fqn)$.
The kernel of $\Phi_\rho$ is the set of all $q$-polynomials which are $\Fq(\rho)$-linear, where $\Fq(\rho)$ is the field extension generated by $\rho$ over $\Fq$.
Indeed, there is a divisor $w$ of $n$ such that 
$\Fq(\rho)=\la1,\rho,\rho^2,\ldots,\rho^{w-1}\ra_{\Fq}$.

If $\Phi_\rho(f)$ is zero, then $f(\rho^kx)=\rho^kf(x)$ for any $k\in\mathbb N$ and this implies
that $f(x)$ is $\Fq(\rho)$-linear.
In particular this can be applied to the following propositions.

\begin{proposition}
For any $\Lt$partially scattered $q$-polynomial $f(x)$ over $\Fqn$ and any $m \in \Fqn$, 
$f(x)+mx$ is $\Lt$partially scattered.
\end{proposition}

\begin{proposition}\label{R-ps+tpolynomial}
	Let $f(x)$ be an $\Rt$partially scattered $q$-polynomial over $\Fqn$.
For 
$\aaa=(a_0,a_1,\ldots,a_{t'-1}) \in \Fqn^{t'}$ define the $q^t$-polynomial
\begin{equation}\label{e:gax}
 g_\aaa(x)=\sum_{i=0}^{t'-1}a_ix^{q^{it}}
\end{equation}
Then for any $\aaa \in \Fqn^{t'}$ the $q$-polynomial $f(x)+g_\aaa(x)$ 
is $\Rt$partially scattered as well.
	\end{proposition}
\begin{proof}
Since for all $\rho \in \Fqt \setminus \Fq$, 
$f_\rho(x)$ is bijective and $(g_\aaa)_{\rho}(x)$ is the zero polynomial, the result follows.
\end{proof}

\begin{sloppypar}
\begin{proposition}\label{p:costruzione}
Let $\aaa=(a_0,a_1,\ldots,a_{t'-1}) \in \Fqn^{t'}$, and
$g_\aaa(x)$ be the $q^t$-polynomial defined in \eqref{e:gax}.
\begin{enumerate}[(a)]
\item
For any $\Rt$partially scattered $q$-polynomial $\varphi(x)$, 
the $q$-polynomial $f(x)=(g_\aaa\circ\varphi)(x)$ is $\Rt$partially scattered if and only if
$g_\aaa(x)$ is bijective.
\item 
Let $s<n$ be a positive divisor of $n$.
If  $t$ and $s$ are relatively prime, and
$g_\aaa(x)$ is an $\Lt$partially scattered bijective polynomial, then 
$f_s(x)=g_\aaa(x^{q^s})$ is both an $\Rt$ and an R-$q^{s}$-partially scattered $q$-polynomial.
\end{enumerate}
\end{proposition}
\end{sloppypar}
\begin{proof}

(a)
By Proposition \ref{p:frho}~(a), it is enough to show that for any $\rho \in \Fqt \setminus \Fq$, the map $f_\rho(x)$ is bijective if and only $g_\aaa$ is. 
This follows from
$$f_\rho(x)=g_\aaa(\varphi(\rho x))-\rho g_\aaa(\varphi(x))=
g_\aaa(\varphi(\rho x))-g_\aaa(\rho \varphi(x))=g_\aaa(\varphi_\rho(x)).$$

(b)
Since $\varphi(x)=x^{q^s}$ is $\Rt$partially scattered, 
also $f_s(x)$ is $\Rt$partially
scattered by part (a). 
Assume
\[
\frac{f_s(y)}y=\frac{f_s(z)}z \quad \mbox{ for }y,z\in\Fqn^*,\ y=\ell z,\ 
\ell\in\F_{q^s}.
\]
Since $f_s(x)$ is bijective, $f_s(y) \not = 0 \not = f_s(z)$, and

\[\frac{f_s(y)}{f_s(z)}=\frac{y}z=\frac{y^{q^s}}{z^{q^s}}.\]

\noindent This implies

\[
\frac{g_\aaa(y^{q^s})}{y^{q^s}}=\frac{g_\aaa(z^{q^s})}{z^{q^s}},
\]
hence $\ell\in\Fqt\cap\F_{q^s}=\Fq$.
\end{proof}

\begin{remark}\label{r:costruzione}
If $\ker g_\aaa$ is not trivial, the polynomial $f_s(x)$ in Proposition \ref{p:costruzione}~(b)
is not $\Rt$partially  scattered. 
Indeed, let $y\in\F^*_{q^n}$ such that $y^{q^s}=u\in \ker g_\aaa\setminus\{0\}$ and consider $\ell$ 
an  element in $\F_{q^t}\setminus \F_{q}$. 
Since $\ker g_\aaa$ is an $\F_{q^t}$-vector space of $\F_{q^n}$, putting $z=\ell y$
$$f_s(y)=g_\aaa(u)=0=g_\aaa(\ell^{q^s}u)=f_s(\ell y)=f_s(z),$$
so $y/z \in \F_{q^t}$, but $y/z \not \in \F_{q}$.
\end{remark}

\subsection{\texorpdfstring{Examples}{Examples}}
As usual, if $\ell$ divides $m$,
\[
\tr_{q^m/q^\ell}(x)=x+x^{q^\ell}+x^{q^{2\ell}}+\cdots+x^{q^{m-\ell}}\quad\mbox{and}\quad
\N_{q^m/q^\ell}(x)=x^{\frac{q^{m}-1}{q^{\ell}-1}}
\]
denote the \emph{trace} and the \emph{norm} of $x\in\F_{q^{m}}$ over $\F_{q^{\ell}}$.

\begin{proposition}\label{1example}
	If $n=2t$, $s\in\mathbb N$ is coprime with $n$, $f(x)=\delta x^{q^s}+x^{q^{t+s}}$ and $\N_{q^{2t}/q^t}(\delta)\neq1$, then 
	$f(x)$ is $\Rt$partially scattered.
\end{proposition}
\begin{proof}
	Let $g_\aaa(x)=\delta x+x^{q^t}$.
	The kernel of $g_\aaa(x)$ is trivial if, and only if, $\N_{q^{2t}/q^t}(\delta)\neq1$.
	So, the assumptions of the Proposition \ref{p:costruzione}~(a) hold with $\varphi(x)=x^{q^s}$.
\end{proof}

\begin{remark}
	If $n=2t$, $\gcd(s,n)=1$ and $f(x)=\delta x^{q^s}+x^{q^{t+s}}$, then
	\[
	f_\rho(x)=\left(\rho^{q^{t+s}}-\rho\right)x^{q^{t+s}}+\delta(\rho^{q^s}-\rho)x^{q^s}.
	\]
Next, assume $\N_{q^{2t}/q^t}(\delta)=d$.
For $\rho\not\in\Fq$, the map $f_\rho$ is not bijective if and only if
\begin{equation}\label{e:otto}
	\left(\frac{\rho^{q^{t+s}}-\rho}{\rho^{q^s}-\rho}\right)^{q^t+1}=d,
	\end{equation}
Therefore, for given $t$ and $s$, the property of a polynomial of the form 
$\delta x^{q^s}+x^{q^{t+s}}\in\F_{q^{2t}}[x]$
to be scattered (or $\Lt$, or $\Rt$partially scattered) only depends on $\N_{q^{2t}/q^t}(\delta)$.
Such polynomials include those investigated in \cite{CMPZ}.
In particular, since in \cite[Theorem 7.2]{CMPZ} it is shown that for
odd $q$, $t=4$ and $\delta=\sqrt{-1}$
the polynomial $f(x)$ is scattered, then $f(x)$ is scattered for any $\delta$ such that 
$\N_{q^{2t}/q^t}(\delta)=-1$.
Computations with GAP show that for $t=4$, $s=1$, $q=3$ or $q=5$, the condition $\N_{q^{2t}/q^t}(\delta)=-1$ 
is also  necessary for $f(x)$ to be scattered.
Write
\begin{equation}\label{e:rho}
\rho=\rho_1+\rho_2,\quad \rho_1\in\F_{q^t},\ \rho_2\in\ker\tr_{q^{2t}/q^t}.
\end{equation}
Combining \eqref{e:otto} and \eqref{e:rho} gives
\begin{equation}\label{e:rho2}
(\rho_1^q-\rho_1)^2=\rho_2^2+\rho_2^{2q}.
\end{equation}
Equations \eqref{e:rho} and \eqref{e:rho2} then have no solutions with $\rho_2\neq0$ in case $t=4$.
On the other hand, solutions for any $q=3,5,7$ and $5\le t\le 8$ have been found by computer.
\end{remark}

\begin{remark}\label{r:controes-b}
	Since $g_\aaa(x)=\delta x+x^{q^t}\in\F_{q^{2t}}[x]$ is $\Lt$partially scattered for any $\delta\in\F_{q^{2t}}$,
	the last examples in the previous Remark satisfy the conditions in Proposition \ref{p:costruzione}~(b),
	but are not $\Lt$partially scattered.
\end{remark}

\begin{proposition}
	If $n=2t$, $s$ is coprime with $n$, $f(x)=\delta x^{q^s}+x^{q^{t+s}}$ and $\N_{q^{2t}/q^t}(\delta)=1$, then 
	$f(x)$ is $\Lt$partially scattered, but it is not $\Rt$partially scattered.
\end{proposition}
\begin{proof}
	Note that if $\N_{q^{2t}/q^t}(\delta)=1$, $\ker (\delta x+x^{q^t})$ is not trivial,
	then $f(x)$ is not  $\Rt$partially scattered (cf. Remark \ref{r:costruzione}).
	
	Next, assume
	\begin{equation}\label{e:pese}
	\frac{f(y)}y=\frac{f(z)}z,\quad y,z\in\F_{q^{2t}}^*.
	\end{equation}
	If $f(y)=f(z)=0$, then $\delta=-y^{q^{t+s}-q^s}=-z^{q^{t+s}-q^s}$, so $(y/z)^{q^{t+s}-q^s}=1$.
	This implies $y/z\in\F_{q^t}$.
	Otherwise $f(z)\neq0$ may be assumed without loss of generality.
	The condition (\ref{e:pese}) implies 
	\[
	\frac yz=\frac{\delta y^{q^s}+y^{q^{t+s}}}{\delta z^{q^s}+z^{q^{t+s}}}.
	\]
	The quotient $y/z$ belongs to $\F_{q^t}$ if and only if $y/z=(y/z)^{q^t}$, that is
	\[
	\frac{\delta+y^{q^{t+s}-q^s}}{\delta+z^{q^{t+s}-q^s}}= \frac{\delta^{q^t}y^{q^{t+s}-q^s}+1}{\delta^{q^t}z^{q^{t+s}-q^s}+1},
	\]
	equivalently:
	\begin{gather*}
	(\delta+y^{q^{t+s}-q^s})(\delta^{q^t}z^{q^{t+s}-q^s}+1)=(\delta^{q^t}y^{q^{t+s}-q^s}+1)(\delta+z^{q^{t+s}-q^s})\quad
	\Leftrightarrow\\
	\left(\delta^{q^{t}+1}-1\right)\left(y^{q^{t+s}-q^s}-z^{q^{t+s}-q^s}\right)=0.
	\end{gather*}
	Then $f(x)$ is $\Lt$partially scattered.
\end{proof}

\begin{proposition}\label{p:cmmz}
	Assume $n=3t$, $s$ coprime with $n$ and $\delta\in\F_{q^{3t}}$.
	If $\tr_{q^{3t}/q^t}(\delta)-\N_{q^{3t}/q^t}(\delta)\neq2$, then
	$f(x)=x^{q^s}+x^{q^{t+s}}+\delta x^{q^{2t+s}}$ is an $\Rt$partially scattered $q$-polynomial.
\end{proposition}
\begin{proof}
	The polynomial $g(x)=x+x^{q^t}+\delta x^{q^{2t}}$ is bijective if, and only if, its Dickson matrix $D_g$  \eqref{dickson} is non-singular, and this is equivalent to
	$\tr_{q^{3t}/q^t}(\delta)-\N_{q^{3t}/q^t}(\delta)\neq2$.
	Then the assertion holds by Proposition \ref{p:costruzione}~(a). 
\end{proof}
\begin{remark}\label{controes-a}
		If $t=2$ and $\delta^2+\delta=1$, then $\delta\in\F_{q^2}$, hence 
$\tr_{q^{3t}/q^t}(\delta)-\N_{q^{3t}/q^t}(\delta)=1+\delta\neq2$.
So, the assumptions of Proposition \ref{p:cmmz} are satisfied for the polynomial $f(x)=x^q+x^{q^3}+\delta x^{q^5} \in \F_{q^6}[x]$.
		If $q$ is odd, then the polynomial $f(x)$ is also L-$q^2$-partially scattered \cite{CsMZ2018,MMZ}.
If $q$ is even, then $f(x)$ provides an example of a polynomial satisfying the assumptions of
Proposition \ref{p:costruzione}~(a) which is not $\Lt$scattered \cite[Theorem 1.1]{MMZ}.
\end{remark}

Next, let $q$ be odd, $t>1$ an integer, and define
\begin{eqnarray*}
  T_s(x)&=&\left(\tr_{q^{2t}/q^t}(x)\right)^{q^s},\\
  U_k(x)&=&\left(x-x^{q^t}\right)^{q^k}\quad\mbox{for $x\in\F_{q^{2t}}$, $s,k=0,1,\ldots,2t-1$.}
\end{eqnarray*}
The $q$-polynomial $f_{s,k}(x)=T_s(x)+U_k(x)\in\F_{q^{2t}}[x]$ is bijective:
indeed, $\im T_s=\ker U_k=\Fqt$,
$ \ker T_s=\im U_k=\ker(\tr_{q^{2t}/q^t})$,
and \[\Fqt\cap\ker(\tr_{q^{2t}/q^t})=\{0\}.\]
\begin{sloppypar}
\begin{proposition}\label{p:nflike}
If $\gcd(s,2t)=1=\gcd(k,2t)$, then $f_{s,k}(x)$ is an $\Rt$partially scattered $q$-polynomial.
\end{proposition}
\end{sloppypar}
\begin{proof}
For $\rho\in\Fqt$, 
\[
  (f_{s,k})_\rho(x)=(\rho^{q^s}-\rho)T_s(x)+(\rho^{q^k}-\rho)U_k(x),
\]
where $(\rho^{q^s}-\rho)T_s(x)\in\Fqt$ and $(\rho^{q^k}-\rho)U_k(x)\in\ker(\tr_{q^{2t}/q^t})$. 
This implies
that $(f_{s,k})_\rho$ is bijective.
The thesis follows by Proposition \ref{p:frho}~(a).
\end{proof}
In the particular case $s=t-k$,
\[
  f_{t-k,k}(x)=x^{q^k}+x^{q^{t-k}}-x^{q^{t+k}}+x^{q^{2t-k}},
\]
which is the polynomial $2\psi^{(k)}_{2t}$  defined in \cite{nf}.
Under the additional assumption that $q\equiv1\pmod4$ in case $t$ is odd, such a polynomial is also
$\Lt$partially scattered \cite{nf}.
Computations with the software GAP show that the polynomials as in Proposition \ref{p:nflike}
are not always scattered.
For instance, they are not for $(n,s,k)=(5,1,3)$ and $q\le9$.

\subsection{Duality}

Consider the non-degenerate symmetric bilinear form $\la \cdot,\cdot \ra$ of $\Fqn$ over $\Fq$ defined by
\begin{equation}\label{e:bilform}
\la x, y \ra= \tr_{q^n/q}(xy).
\end{equation}
The adjoint map $\hat{f}(x)$ of an $\Fq$-linear map $f(x) = \sum_{i=0}^{n-1}a_ix^{q^i}$ 
with respect to the bilinear form \eqref{e:bilform} 
is 
\begin{equation}\label{c:hatf}
\hat{f}(x) =\sum_{i=0}^{n-1}a^{q^{n-i}}_ix^{q^{n-i}}.
\end{equation}

Let $\eta : \V \times \V \rightarrow \Fq$ be the non-degenerate alternating bilinear form of $\V=
\Fqn^2$, seen as a vector space  over $\Fq$, defined by 
$$\eta((x_1,y_1),(x_2,y_2))  = \tr_{q^n/q}(x_1y_2-x_2y_1)$$
and denote by $\perp$ the orthogonal complement map defined by $\eta$ on the
lattice of the $\Fq$-subspaces of $\V$.
It is straightforward to see that $U_f^\perp=U_{\hat{f}}$.

\begin{sloppypar}
\begin{theorem}[\cite{BaGiMaPo} Lemma 2.6,  \cite{CMP} Lemma 3.1]\label{e:adjline}
Let  $f(x)$ be a  $q$-polynomial over $\Fqn$. %
Then  $w_{L_f}(P)=w_{L_{\hat{f}}}(P)$ for any $P \in \PG(1, q^n)$. 
In particular, $L_f = L_{\hat{f}}$ and the maps defined by $f(x)/x$ and $\hat{f}(x)/x$ have the same image.
\end{theorem} 
\end{sloppypar}

\begin{sloppypar}
\begin{proposition} 
	Let $f(x)$ be a $q$-polynomial over $\F_{q^n}$.
	Then $f(x)$ is $\Lt$partially scattered ($\Rt$partially scattered) if and only if $\hat{f}(x)$ is 
	$\Lt$partially scattered ($\Rt$partially scattered). 
\end{proposition}

\begin{proof}
Since the adjoint map $f(x) \mapsto \hat{f}(x)$ is involutory, it is enough to prove sufficiency. 
Assume $y,z,\lambda\in\F_{q^n}^*$ and
\begin{equation*}
\frac{\hat f(y)}y=\frac{\hat f(z)}z,\quad\lambda y=z,
\end{equation*}
then $\hat{f}(\lambda y)=\lambda \hat{f}(y)$.
By applying the adjoint map to both sides of the equation above, $\lambda f(y)=f(\lambda y)$, and
\begin{equation*}
\frac{f(y)}y=\frac{f(z)}z.
\end{equation*}
This is enough to deduce that $\hat{f}(x)$ is $\Lt$partially scattered (resp.\ $\Rt$partially scattered)
if $f(x)$ is $\Lt$partially scattered (resp.\ $\Rt$partially scattered).
\end{proof}
\end{sloppypar}

\begin{remark}
As mentioned in Proposition \ref{e:adjline}, if $f(x)$ is a linearized polynomial over $\Fqn$ and $\hat{f}(x)$ is its adjoint, then  $L_f=L_{\hat{f}}$.  
On the other hand, the linear sets $R_f$ and $R_{\hat{f}}$ in $\PG(2t'-1,q^t)$ do not necessarily coincide.
For example, consider $f(x)=x^q$ and suppose that $R_f=R_{\hat{f}}$. 
Then for any $x \in \Fqn^*$ there exists $y \in \Fqn^*$ such that
$$\la (x,x^q)\ra_{\Fqt}=\la(y,y^{q^{n-1}}) \ra_{\Fqt},$$
hence one gets that there exists $\lambda \in \Fqt$ such that 
$$\lambda^{q^{n-1}-1}=x^{q-q^{n-1}}.$$
Since the left hand side of the equation above is an element of $\Fqt$, 
$$(x^{q-q^{n-1}})^{q^t-1}=1,$$
for all $x \in \Fqn^*$, and this is equivalent to
\begin{equation*}
z^{(q^{n-2}-1)(q^t-1)}=1
\end{equation*}
and hence, modulo $z^{q^n}-z$,
\begin{equation*}
z^{(q^{n-t}-1)(q^2-1)}=1
\end{equation*}
for all $z \in \Fqn^*$. So, $q^{n}-1$ divides  $(q^{n-t}-1)(q^2-1)$, 
obtaining a contradiction since $t \geq 2$.
\end{remark}

\section{\texorpdfstring{Partially scattered $q$-polynomials and MRD-codes}%
{Partially scattered q-polynomials and MRD-codes}}\label{s:tre}

 
As retraced in Section \ref{s:uno}, each element of $\End_{\F_q}(\F_{q^n})$	can be represented 
as a unique linearized polynomial over $\F_{q^n}$ of $q$-degree less than $n$ and vice versa. Then, the set $\tilde{\mathcal{L}}_{n,q}$ of these $q$-polynomials with the usual sum, scalar multiplication and  product $\circ$, the functional composition modulo $x^{q^n}-x$, is an algebra over $\Fq$ isomorphic to the algebra of square matrices $\F_{q}^{n \times n }$.
Hence, any RM-code (resp.\ $\Fq$-linear RM-code)  $\cC \subset \F_q^{n \times n}$ might be regarded as a 
suitable subset  (resp.\ subspace) of $\tilde{\mathcal{L}}_{n,q}$ and any 
related notion or property may 
be translated in the setting of linearized polynomials.
Moreover, it is straighforward to see that $\{\hat{f} \colon f \in \cC\}$ (cf. \eqref{c:hatf}) is the adjoint code $\cC^\top$ of $\cC$ and, considered $b : \tilde{\mathcal{L}}_{n,q} \times \tilde{\mathcal{L}}_{n,q} \rightarrow \Fq$ the bilinear form given by
$$b(f, g) = \tr_{q^n/q} \Biggl (\sum^{n-1}_{i=0}f_ig_i \Biggr )$$
where $f(x) =\sum^{n-1}_{i=0}f_ix^{q^i}$ and $g(x) =\sum^{n-1}_{i=0} g_ix^{q^i}$,
the Delsarte dual code of $\cC$ is 
$$\cC^\perp =\{f \in \tilde{\mathcal{L}}_{n,q} \colon b(f,g)=0, \forall g \in \cC \}.$$
In  \cite[Section 5]{Sh}, Sheekey explicated a link between  maximum scattered $\Fq$-linear sets of 
$\PG(1, q^n)$ and $\Fq$-linear MRD-codes with minimum distance $d=n-1$.
In contrast to previous research, in this section we will focus on codes of non-square matrices $n \times t$, where $t$ is a divisor of $n$.

\begin{proposition}\label{p:lat-rps}
\begin{enumerate}[(a)]
Let $f(x)$ be $\Rt$partially scattered. Then
\item
In $V(\Fqn,\Fq)$, $\ker f$ is a scattered $\Fq$-subspace with respect to the normal
$t$-spread $\cN=\{\la v\ra_{\Fqt}\colon v\in\Fqn^*\}$; that is, $\dim(\ker f \cap \la v \ra_{\Fqt}) \leq 1$ for any $v \in \F^*_{q^n}$.
\item
$\dim_{\Fq}\ker f\le n/2$.
\item
There exists $\aaa=(a_0,a_1,\ldots,a_{t'-1})\in\Fqn^{t'}$ such that the $\Fq$-dimension of the 
kernel of 
$f(x)+g_\aaa(x)$ (cf.\ \eqref{e:gax}) is at least $t'$.
\end{enumerate}
\end{proposition}
\begin{proof}
(a)
If $y,z\in\ker f\cap \la v\ra_{\Fqt}$ and $y,z,v\in\Fqn^*$, then
\eqref{e:erre} applies, hence $\dim_{\Fq}\ker f\cap \la v\ra_{\Fqt}=1$. 

(b)
This follows from the previous assertion and \cite[Theorem 4.3]{BL2000}.

(c)
Let $v_1$, $v_2$, $\ldots$, $v_{t'}$ be an $\Fqt$-basis of $\Fqn$.
Then there exists a $q^t$-polynomial $g_\aaa(x)=\sum_{i=0}^{t'-1}a_ix^{q^{ti}}$ such that
$
g_\aaa(v_j)=-f(v_j)$, $j=1,2,\ldots,t'
$.
One obtains a $q$-polynomial
$f(x)+g_\aaa(x)$ vanishing on $t'$ $\Fq$-linearly independent elements of $\Fqn$.
\end{proof}

In case $t$ and $t'$ are relatively prime and $f(x)$ is an $\Lt$partially scattered polynomial,
then $f(x)$ is R-$q^{t'}$-partially scattered, and Proposition \ref{p:lat-rps} can be applied.
The question arises of what the dimension of the kernel of an $\Rt$partially scattered polynomial
can be.

Gathering the results in Proposition \ref{p:costruzione}~(a) and Proposition \ref{p:lat-rps}~(c), one obtains that any polynomial of the form
\begin{equation}\label{R-ps-shape}
g_\aaa(x^\gamma)-mx^{q^{kt}},
\end{equation}
where $g_\aaa(x)$ is a bijective $q^t$-polynomial, $\gamma$ a generator of $\Gal(\Fqt / \Fq)$, 
$m \in \Fqn^*$ and $0 \leq k \leq t'-1$,
is an $\Rt$partially scattered polynomial.
Studying the size of the kernel of a  polynomial  as in \eqref{R-ps-shape} is equivalent to analyzing the kernel of
$$x^{q^{n-kt}} \circ (g_\aaa(x^\gamma)-mx^{q^{kt}}),$$
i.e.\ of a polynomial of the form
$$g_{\mathbf{b}}(x^\gamma)-\ell x,$$
where $g_{\mathbf{b}}(x)$ is an appropriate bijective $q^t$-polynomial and $\ell \in  \Fqn^*$.
Finally, since 
$$\ker( g_\mathbf{b}(x^\gamma)-\ell x) = \ker(\ell^{-1}g_\mathbf{b}(x^\gamma)- x),$$
by Theorem 1.5, assertion 1., in \cite{OlgaFer2}, the $\Rt$partially scattered polynomials as in 
\eqref{R-ps-shape} have kernel of dimension at most $t'$.
In \cite{delacruz} an \emph{$s$-almost MRD-code} is an RM-code with minimum distance $d=d^*-s$,
where $d^*$ denotes the maximum allowed by the Singleton-like bound.
By the above arguments the following holds:
\begin{theorem}
Assume that the map $g_\aaa(x)$ defined in \eqref{e:gax} is bijective.
Let $\gamma$ be a generator of $\Gal(\F_{q^t} / \Fq)$, and $k\in\mathbb N$.
Then $\la g_{\aaa}(x^{\gamma}),x^{q^{kt}}\ra_{\Fqn}$ is an $s$-almost MRD-code in
$\F_q^{n\times n}$ with $s\le n/t-1$.
\end{theorem}

Up to the knowledge of the authors of this paper, no $\Rt$partially scattered
polynomial is known having a kernel with $\Fq$-dimension greater than $t'$.
Given  a $q$-polynomial $f(x)$ over $\Fqn$ and an element $\sigma \in \Fqn^*$, 
consider the following $\Fq$-linear map:
	$$f_{t,\sigma}: x \in \F_{q^{t}} \mapsto f(\sigma x) \in \Fqn.$$

\begin{sloppypar}
\begin{theorem}\label{R-ps_MRDsquare}
	Let $f(x)$ be a $q$-polynomial over $\Fqn$.
	\begin{enumerate}[(a)]
		\item  Let $k$ be a non-negative integer. Then $f(x)$ is an $\Rt$partially scattered polynomial if and only if 
\begin{equation}\label{e:cfsk}
\cC_{f,\sigma,k}=\la f_{t,\sigma}(x),x^{q^{kt}}\ra_{\Fqn}
\end{equation}
is an MRD-code with parameters $(n,t,q;t-1)$ for any $\sigma \in \Fqn^*$.
		
		\item
		Assume $t=t'=2$. A $q$-polynomial $f(x)$ over $\F_{q^4}$ is $\Rt$partially scattered if and only if \begin{equation}\label{e:c2x2}
		\cC=\la f(x),x,x^{q^{2}}\ra_{\F_{q^4}}
\end{equation}
is an MRD-code with parameters $(4,4,q;2)$.
	\end{enumerate}
\end{theorem}

\begin{proof}
(a)
	Suppose that $f(x)$ is $\Rt$partially scattered and consider $\sigma \in \Fqn^*$. 
	By Proposition \ref{R-ps+tpolynomial}, the polynomials $f(\sigma x)+mx^{q^{kt}}$, 
	$m\in \F_{q^n}$, are all $\Rt$partially scattered, hence their restrictions to $\Fqt$ have 
	rank at least $t-1$.
		The code $\cC_{f,\sigma,k}$ has $\Fq$-dimension $2n$ and
		\[
		\max\{t,n\}(\min\{t,n\}-(t-1)+1)=2n,
		\]
		that is, it is MRD.
		
		Now suppose that for any $\sigma \in \Fqn^*$, $\cC_{f,\sigma,k}=\la f_{t,\sigma}(x),x^{q^{kt}}\ra_{\Fqn}$ is an MRD-code with parameters as in the statement, and let 
		\begin{equation}
		\frac{f(u)}{u}=\frac{f(v)}{v}=-m
		\end{equation}
		with $v=\lambda u$, for some $\lambda \in \Fqt$, and $m,u,v\in \Fqn$, $u\neq0\neq v$.
		Since
		$$f(u)+mu=f(\lambda u)+mu\lambda=0,$$
		we get that $1,\lambda \in \ker (f_{t,u}(x)+mux^{q^{kt}})$. 
		Therefore, since $f_{t,u}(x)+mux^{q^{kt}}\in \cC_{f,u,k}$, 
		its kernel is an $\Fq$-subspace of rank at most one, and $\lambda$ belongs to $\Fq$.
	
(b)
	The dimension over $\Fq$ of $\cC$ is twelve. 
		Taking into account Proposition \ref{p:lat-rps}~$(b)$, the equation in the 
		Singleton-like bound holds.
		
		Vice versa, let $\cC=\la f(x),x,x^{q^{2}}\ra_{\F_{q^4}}$ be an MRD-code with parameters $(4,4,q;2)$ and, without loss of generality, suppose that $f(x)=a_1x^q+a_3x^{q^3}$. 
		Then, since the Delsarte dual code $\cC^\perp$  of $\cC$ is an MRD-code with parameters $(4,4,q;4)$ (also known as \textit{semifield spread set}), the map 
		$$f^\perp: x \in \F_{q^4} \mapsto a_3x^q-a_1x^{q^3} \in \F_{q^4}$$ 
		is an $\F_q$-automorphism of $\F_{q^4}$. It is a straightforward calculation to show that the determinant of its Dickson matrix  $D_{f^{\perp}}$ is 
		\begin{equation}\label{det}
		-(a^{q^2+1}_3-a_1^{q^2+1})^{q+1},
		\end{equation}
		hence $a_3^{q^2+1}\neq a_1^{q^2+1}$.
		This implies that
		$g(x)=a_1x+a_3x^{q^2}$ is an automorphism of $\F_{q^4}$. 
		By Proposition \ref{p:costruzione}, $f(x)=g_{(a_1,a_3)}(x^q)$ is R-$q^2$-partially scattered.
\end{proof}

Note that by Proposition \ref{p:facile}, if $t$ and $t'$ are relatively prime, 
any $\Lt$scattered $q$-polynomial $f(x)$ gives rise to an MRD-code 
$\cC'_{f,\sigma,k}=\la f_{t',\sigma}(x),x^{q^{kt'}}\ra_{\Fqn}$ for any $\sigma\in\Fqn^*$ and any
$k\in\mathbb N$.

\end{sloppypar}

\begin{proposition}\label{p:ideal22}
Assume $t=t'=2$ and let $f(x)$  be  $\Rt$partially scattered. 
Then 
$$L(\cC)=R(\cC)\cong \F_{q^4},$$
where $\cC$ is as in \eqref{e:c2x2}.
Furthermore, $\cC$ is equivalent to a generalized Gabidulin code.
\end{proposition}
\begin{proof} Let $f(x)$ be an $\Rt$partially scattered $q$-polynomial and, without loss of generality, suppose that 
$f(x)=a_1x^q+a_3x^{q^3}$. 
As in Theorem \ref{R-ps_MRDsquare}, consider the automorphism 
$f^{\perp}(x)=a_3x^q-a_1x^{q^3}$. An element $\varphi(x)=\sum^3_{i=0}\varphi_ix^{q^i}$ belongs to 
$L(\cC^\perp)$ if and only if composing $\varphi(x)$ to
the left with an $f^\perp(x)$'s $\Fqn$-multiple   is still an  
$f^\perp(x)$'s $\Fqn$-multiple. 
Then, for any $\beta \in \Fqn^*$ there exists $\alpha \in \Fqn$ such 
that $\varphi(x) \circ \beta f^\perp(x)= \alpha f^\perp(x)$.
	Making more explicit the latter condition, we get that the
coefficients of $\varphi(x)$ have to 
	solve the following  linear systems
	\begin{center}
		\begin{minipage}{0.41\textwidth}
			\begin{equation*}
				\begin{cases}
					a_3 \beta	\varphi_0-(a_1\beta)^{q^2} \varphi_2 =\alpha a _3 \\
					a_1\beta \varphi_0 - (a_3\beta)^{q^2}\varphi_2=\alpha a_1\\
				\end{cases}
			\end{equation*}
		\end{minipage}
		\hspace{15mm}
		\begin{minipage}{0.41\textwidth}
			\begin{equation*}
				\begin{cases}
					 (a_3\beta)^q	\varphi_1-(a_1\beta)^{q^3} \varphi_3 =0\\
					(a_1\beta)^q\varphi_1 - (a_3\beta)^{q^3}\varphi_3=0.\\
				\end{cases}
			\end{equation*}
		\end{minipage}\end{center}

Since the expression in \eqref{det} is not zero, it is a straightforward calculation to show that 
$\varphi_0=\alpha\beta^{-1}$ and $\varphi_i=0$, $i=1,2,3$, that is $\varphi(x)=\gamma x$, $\gamma \in \Fqn$. 
Since $L(\cC^\perp)=L(\cC)^\top$ and $\hat{\varphi}(x)=\varphi(x)$, $L(\cC)\cong \Fqn$.

	Consider now the code $\cD=\la -a_1^qx^q+a_3^{q^3}x^{q^3} \ra_{\F_{q^4}}$, the adjoint code of $\cC^\perp$. Applying the same argument above, we get that $L (\cD)=\{\gamma x \,\, |\,\, \gamma \in \Fqn\}$. Since the maps in $L(\cD)$ are self-adjoint and the $\top$ operator is involutory, 
	$$\Fqn \cong L(\cD)=R(\cC^\perp)^\top=R(\cC),$$ 
	then the result follows.
	
In \cite[Theorem 1.1]{CMPZideal} it is shown that for $2\le n\le 6$ or $n=9$, any 
$\Fq$-linear MRD-code in $\F_q^{n\times n}$ with left and right idealisers isomorphic to $\Fqn$
is equivalent to a generalized Gabidulin code.
This implies the second assertion of the proposition.
\end{proof}

\begin{proposition}\label{p:rightpsidealiser}
Let $g_\aaa(x)$ be a  bijective $q^t$-polynomial as in \eqref{e:gax}, $t>2$. 
Consider $0<s<n$ an integer relatively prime with $t$, $\sigma\in\F_{q^n}^*$, and $k \in \mathbb{N}$. 
Then 
$$R(\cC_{h,\sigma,k}) \cong \Fqt,$$ 
where $h(x)=g_\aaa(x^{q^s})$.
\end{proposition}

\begin{proof} By Proposition \ref{p:costruzione}~(a) and \ref{R-ps_MRDsquare}~(a), 
$\cC_{h,\sigma,k}=\la h_{t,\sigma}(x),x^{q^{kt}} \ra_{\F_{q^n}}$ is an MRD-code with parameters $(n,t,q;t-1)$ 
and, by Proposition \ref{MRD-idealisers}, the right idealiser
 $$R(\cC_{h,\sigma,k})=\{\varphi \in \End_{\Fq}(\Fqt) : f(\varphi(x))\in \cC_{h,\sigma,k}\,\,\forall f \in 
 \cC_{h,\sigma,k}\}$$ 
is isomorphic to a field of size at most $q^t$.
Since
\[h(\sigma x)=\left(\sum_{i=0}^{t'-1}a_i \right)(\sigma x)^{q^s} \,\,\mod (x-x^{q^t}),
\] 
it follows that
$\cC_{h,\sigma,k}=\la x^{q^s},x^{q^{kt}} \ra_{\F_{q^n}} $. 
Hence if $\varphi(x)=\sum_{i=0}^{t-1}\varphi_ix^{q^{i}}$ belongs to $R(\cC_{h,\sigma,k})$, then there  exist 
$a, b \in \Fqn$ such that
$$\varphi(x)^{q^{kt}}=\varphi(x)=ax^{q^s}+bx$$
modulo $x^{q^t}-x$,
hence $\varphi_i=0$ for all $i \not = 0,s$. 
Similarly, $\varphi(x)$ is in $R(\cC_{h,\sigma,k})$ if there exists $c,d \in \Fqn$ such that
$$\varphi(x)^{q^s}=cx^{q^s}+dx,$$
hence one gets that $\varphi_s=0$ and $\varphi(x)=\gamma x$ with $\gamma \in \Fqt$. It a straightforward calculation to see that the $\Fq$-endomorphisms of $\Fqt$ $\varphi(x)=\gamma x $ with $\gamma \in \Fqt$ are in $R(\cC_{h,\sigma,k})$. 
Therefore, the result follows.
\end{proof}

Clearly the examples discussed in Propositions \ref{1example} and \ref{p:cmmz} satisfy the hyphoteses of proposition above.

\begin{remark}
Note that in general, if $f(x)$ is any $\Rt$partially scattered polynomial, the MRD-codes 
$\cC_{f,\sigma,k}$ in 
Theorem~\ref{R-ps_MRDsquare}~(a) could have a right idealiser not isomorphic to $\Fqt$. 
Indeed, consider the polynomial $f(x)=\delta x^q+x^{q^{2t-1}} \in \F_{q^{2t}}[x]$ with 
$\N_{q^n/q}(\delta) \not \in \{0,1\}$,  and $t \geq 5$. 
Since $f(x)$ is scattered, it is $\Rt$partially scattered. 
As in Proposition \ref{p:rightpsidealiser}, an $\Fq$-endomorphism 
$\varphi(x)=\sum_{i=0}^{t-1}\varphi_ix^{q^i}$ of $\Fqt$  is in $R(\cC_{f,1,k})$, 
if there exist $a,b \in \Fqn$ such that
$$\varphi(x)=a(\delta x^q+x^{q^{t-1}})+bx,$$
hence $\varphi_i=0$ for all $i \not \in \{0,1,t-1\}$. 
Similarly, $\varphi(x)$ belongs to $R(\cC_{h,1,k})$ if there exist $c,d \in \Fqn$ such that

$$\delta \varphi(x)^q+\varphi(x)^{q^{t-1}}=c(\delta x^q+x^{q^{t-1}})+dx.$$
Reducing modulo $x-x^{q^t}$ the expression above, $\varphi_1=\varphi_{t-1}=0$ and $\varphi_0 \in \F_{q^2}$. On the other hand, it is trivial that the maps $\varphi(x)= \gamma x $ with $\gamma  \in \F_{q^2}$, if $t$ is even, and $\gamma \in \Fq$, if $t$ is odd are in $R(\cC_{f,1,k})$. Then 
\begin{equation*}
R(\cC_{f,1,k})\cong
\begin{cases}
\F_{q^2} \quad \textnormal{if $t$ is even} \\
\Fq  \,\quad  \textnormal{if $t$ is odd}
\end{cases}.
\end{equation*}
\end{remark}

\begin{remark}
	In \cite{BaZh} some properties and generalizations of the scattered linearized polynomials are
	dealt with which may suggest further approaches to the partially scattered linearized
	polynomials.
	In particular, in \cite[Lemma 2.1]{BaZh} it is noted that
	a $q$-polynomial $f(x)$ is scattered if and only if the algebraic curve $\mathcal V$ in $\PG(2,q^n)$
	defined by
	\[
	\frac{f(X)Y-Xf(Y)}{XY^q-X^qY}
	\]
	(that is indeed a polynomial)
	contains no affine point $(x,y)$ such that $x/y\not\in\Fq$.
	Clearly $f(x)$ is L-$q^t$-partially scattered (resp.\ R-$q^t$-partially scattered) if and only if
	$\mathcal V$ contains no affine point $(x,y)$ such that $x/y\not\in\F_{q^t}$
	(resp.\ such that $x/y\in\F_{q^t}\setminus\Fq$).
\end{remark}



\textbf{Acknowledgement}

The authors are grateful to Olga Polverino for careful reading of the paper and valuable suggestions and 
comments.

\noindent
Giovanni Longobardi and Corrado Zanella\\
Dipartimento di Tecnica e Gestione dei Sistemi Industriali\\
Universit\`a degli Studi di Padova\\
Stradella S. Nicola, 3\\
36100 Vicenza VI - Italy\\
\emph{\{giovanni.longobardi,corrado.zanella\}@unipd.it}

\end{document}